\documentclass[11pt,leqno]{amsart}

\usepackage{latexsym,amsmath,amsfonts,amscd}

\newtheorem{theorem}{Theorem}[section]
\newtheorem{lemma}[theorem]{Lemma}
\newtheorem{proposition}[theorem]{Proposition}
\newtheorem{corollary}[theorem]{Corollary}
\newtheorem{definition}[theorem]{Definition}
\newtheorem{definitions}[theorem]{Definitions}

\newtheorem{remark}[theorem]{Remark}

\def\r{\mathbb{R}}
\def\rn{\mathbb{R}^N}

\def\eps{\varepsilon}
\def\rh{\rightharpoonup}
\def\irn{\int_{\rn}}

\def\vp{\varphi}

\def\ol{\overline}
\def\la{\lambda}
\def\la1{\lambda_1}
\def\d12{\mathcal{D}^{1,2}}
\def\cn{\mathcal{N}}

\numberwithin{equation}{section}

\title[A logarithmic Schr\"{o}dinger equation]{A logarithmic Schr\"{o}dinger equation with asymptotic conditions on the potential}

\author{Chao Ji}\thanks{The first author was supported by the China Scholarship Council and NSFC (grant No. 11301181) and the Fundamental Research Funds for the Central Universities. \\
\indent This work was carried out when the first author visited Stockholm University. He is very grateful to the members of the Department of Mathematics for their warm hospitality.}
\address{Department of Mathematics, East China University of Science and Technology,  200237 Shanghai, China}
\email{jichao@ecust.edu.cn}

\author{Andrzej Szulkin} 
\address{Department of Mathematics, Stockholm University, 106 91  Stockholm, Sweden}
\email{andrzejs@math.su.se}



\begin{document}

\baselineskip15pt

\maketitle

\begin{abstract}
In this paper we consider a class of logarithmic Schr\"{o}dinger equations with a potential which may change sign. When the potential is coercive, we obtain infinitely many solutions by adapting some arguments of the Fountain theorem, and in the case of bounded potential we obtain a ground state solution, i.e.\ a nontrivial solution with least possible energy. The functional corresponding to the problem is the sum of a smooth and a convex lower semicontinuous term.
\end{abstract}

\section{Introduction} \label{intro} 

In this paper we consider the logarithmic Schr\"{o}dinger equation 
 \begin{equation}  \label{1}
-\Delta u+V(x)u =u \log u^{2}, \quad  x\in \rn,
\end{equation}
where the potential $V$ is continuous and satisfies the asymptotic condition $\lim_{|x|\to\infty}V(x) = V_\infty$, where $V_\infty$ is a constant such that  $V_\infty +1\in (0,\infty]$. Applications of this kind of equations to different problems in physics have been discussed in \cite{r12}, see also \cite{r17}. The mathematical literature concerning the logarithmic Schr\"odinger equation does not seem to be very extensive, let us here mention \cite{r4, r5, r6, r7, r9, r12}. 

The energy functional $J$ associated with problem \eqref{1} is
\begin{equation} \label{fcl} 
J(u):= \frac{1}{2}\irn\left(\vert \nabla u\vert^{2}+(V(x)+1)u^{2}\right)dx-\frac{1}{2}\irn u^{2}\log u^{2}\,dx
\end{equation}
and it is easy to see that, formally, each critical point of $J$ is a solution of \eqref{1}. However, this functional takes the value $+\infty$ for some $u\in H^1(\rn)$ and in particular, it is not of class $C^1$. This will force us to go beyond standard critical point theories. More precisely, we will use the critical point theory developed in \cite{r13}, see Section \ref{prel}. 

Now we state our main results. Denote the spectrum of $-\Delta + V$ in $L^2(\rn)$ by $\sigma(-\Delta+V)$. 

\begin{theorem} \label{thm1}
If $V\in C(\rn,\r)$ and $\lim_{|x|\to\infty}V(x)=\infty$, then equation \eqref{1} has infi\-nitely many solutions $\pm u_n$ such that $J(\pm u_n)\to\infty$.
\end{theorem}

\begin{theorem} \label{thm2}
If $V\in C(\rn,\r)$, $\lim_{|x|\to\infty}V(x) = \sup_{x\in\rn}V(x) := V_\infty \in (-1,\infty)$ and the spectrum $\sigma(-\Delta+V+1)\subset (0,\infty)$, then equation \eqref{1} has a ground state solution $u>0$. 
\end{theorem}

In Theorem \ref{thm1} we shall work in the space 
\begin{equation} \label{space}
X := \{u\in H^1(\rn): \irn(|\nabla u|^2+(V(x)+1)^+u^2)\,dx < \infty\},
\end{equation}
where $V^\pm := \max\{\pm V,0\}$. Under the assumptions of this theorem $X$ is compactly embedded in $L^2(\rn)$ and hence (by interpolation) in $L^p(\rn)$ for all $2\le p<2^*$, where $2^* := 2N/(N-2)$ if $N\ge 3$ and $+\infty$ if $N=1$ or 2. This result is well known and can be found (implicitly) e.g.\ in \cite{r11}. It implies in particular that $\sigma(-\Delta+V+1)$ consists of eigenvalues converging to infinity and the quadratic form $u\mapsto \irn(|\nabla u|^2+(V(x)+1)u^2)\,dx$ is positive definite on a space of finite codimension. Some weaker conditions ensuring that the embedding $X\hookrightarrow L^2(\rn)$ is compact (and which would in fact suffice for the conclusion of Theorem \ref{thm1} to hold) have been discussed in detail in \cite{bpw}.

We also note that depending on the choice of $V$ the functional $J$ may or may not take the value $+\infty$ in $X$. Indeed, if $N=1$ and $u$ is a smooth function such that $u(x) = (\sqrt x\log x)^{-1}$ for $x\ge 2$ and $u(x)=0$ for $x\le 0$, then $u\in X$ and $J(u)=+\infty$ provided $V$ has sufficiently slow growth, e.g.\ $V(x) = (\log x)^{1/2}$ for $x\ge 2$. On the other hand, if $V\ge \delta > 0$ and $1/V$ is integrable on $|x|>R$ for some $R>0$, then, using that $|u\log u| \le C_q(1+|u|^q)$ for any $q>1$, we obtain by the H\"older inequality
\begin{align*}
\left|\int_{|x|>R}u^2\log u^2\,dx \right| & \le \left(\int_{|x|>R}V(x)u^2\,dx\right)^{1/2} \left(\int_{|x|>R}V(x)^{-1}u^2(\log u^2)^2\,dx\right)^{1/2} \\
& \le  C\left(\int_{|x|>R}V(x)u^2\,dx\right)^{1/2} \left(\int_{|x|>R}V(x)^{-1}(1+|u|^{2q})\,dx\right)^{1/2} < \infty
\end{align*}
($q$ should be chosen in $(1,2^*/2)$). Since $\int_{|x|<R}u^2\log u^2\,dx$ is finite by the growth condition on $u\log u$, $J$ takes only finite values in this case.  

In Theorem \ref{thm2} we shall work in the space $H^1(\rn)$. Although we do not assume $V+1$ is positive everywhere, we do assume that $\sigma(-\Delta+V+1)\subset (0,\infty)$, or equivalently, that the quadratic form $u\mapsto \irn(|\nabla u|^2+(V(x)+1)u^2)\,dx$ is positive definite on $H^1(\rn)$. It would be interesting to see if one can remove this assumption at the expense of prescribing some asymptotic conditions on $V$ at infinity as in  \cite{cds} or \cite{ew}. We remark that if $V(x)=V_\infty>-1$ for all $x$, then Theorem \ref{thm2} is a special case of Theorem 1.2 in \cite{r12}.

The rest of the paper is organized as follows. In Section \ref{prel} we summarize some pertinent results and definitions, mainly taken from \cite{r12}. In Section \ref{pf1} we prove Theorem \ref{thm1} and also indicate how a minor gap in a proof in \cite{r12} can be removed (see Remark \ref{cptsupp} below). This gap has been pointed out to the authors of \cite{r12} by Chengxiang Zhang. In Section  ~\ref{pf2} Theorem \ref{thm2} is proved and in the final Section \ref{plapl} we briefly sketch how the results of Theorems \ref{thm1} and \ref{thm2} can be generalized to an equation involving the $p$-Laplacian.  

\bigskip

\noindent \textbf{Notation.} $C, C_1, C_2$ etc. will denote positive constants whose exact values are inessential. $\langle .\, , . \rangle$ is the duality pairing between $X^*$ and $X$, where $X$ is a Banach space and $X^*$ its dual. When $X$ is a Hilbert space (which will be the case most of the time), then $\langle .\, , . \rangle$ is the inner product and $X^*$ will be identified with $X$ via duality. $\|\cdot\|_p$ denotes  the norm of the space $L^p(\rn)$.\  $2^*:=2N/(N-2)$ if $N\geq 3$ and $2^*:=\infty$ if $N=1$ or $2$.\ 
$B_R(x)$ denotes the open ball of radius $R$ and center at $x$. 
For a functional $J$ on $X$ we set $J^b:=\{u\in X: J(u)\le b\}$, $J_a:=\{u\in X: J(u)\ge a\}$, $J_a^b:=J_a\cap J^b$. The set of critical points of $J$ (to be defined in the next section) will be denoted by $K$ and we put $K_d := K\cap J_d^d$.

\section{Preliminaries} \label{prel}

As we have seen in the introduction, the functional $J$ may take the value $+\infty$. Since for $|u|\ge 1$ we have $0\le u^2\log u^2 \le C(u^2+|u|^p)$, where $p$ may be chosen in $(2,2^*)$, $J$ can never take the value $-\infty$. Below $X$ will denote the space \eqref{space} with inner product
\begin{equation} \label{inner1}
\langle u,v\rangle := \irn(\nabla u\cdot \nabla v+(V(x)+1)^+uv)\,dx
\end{equation}
when Theorem \ref{thm1} is considered and $X=H^1(\rn)$ with inner product
\begin{equation} \label{inner2}
\langle u,v\rangle := \irn(\nabla u\cdot \nabla v+(V(x)+1)uv)\,dx
\end{equation}
when $V$ is as in Theorem \ref{thm2}.
Note that the positivity assumption on the spectrum of $-\Delta+V+1$ in this theorem implies that \eqref{inner2} is indeed an inner product, equivalent to the usual one in $H^1(\rn)$. 

By a solution to \eqref{1} we mean a function $u\in X$ such that $u^2\log u^2\in L^1(\rn)$ (i.e., $J(u)<\infty$) and 
\[
\irn(\nabla u\cdot \nabla v + V(x)uv)\,dx = \irn uv\log u^2\,dx \quad \text{for all } v\in C_0^\infty(\rn).
\]
Local estimates and standard bootstrap arguments show that such $u$ is a classical solution.

As in \cite{r12}, we set 
\begin{equation*}
F_{1}(s):=\left\{
\begin{array}{ll}
-\frac{1}{2}s^{2}\log s^{2}, &\vert s\vert<\delta,\\
-\frac{1}{2}s^{2}(\log \delta^{2}+3)+2\delta\vert s\vert-\frac{1}{2}\delta^{2}, & \vert s\vert>\delta,
\end{array}%
\right.
\end{equation*}%
and $F_{2}(s):=\frac{1}{2}s^{2}\log s^{2}+F_{1}(s)$. Then $F_2(s)-F_1(s) = \frac12s^2\log s^2$ and taking a sufficiently small $\delta>0$, $F_1$ is convex, $F_1,F_2\in C^1(\r,\r)$ and since $F_2(s)=0$ for $|s|<\delta$, $|F_2'(s)|\le C_p|s|^{p-1}$, where $p$ can be chosen arbitrarily in $(2,2^*)$. Let us now define
\begin{align*}
&\Phi(u):=\frac{1}{2}\irn\left(\vert \nabla u\vert^{2}+(V(x)+1)u^{2}\right)dx - \irn F_{2}(u)\,dx,\nonumber\\
&\Psi(u):=\irn F_{1}(u)\, dx.\nonumber
\end{align*}
Then $J(u) = \Phi(u)+\Psi(u)$, $\Phi\in C^1(X,\r)$ (by the growth condition on $F_2'$, see \cite{r16}), $\Psi\ge 0$, $\Psi$ is convex and, by Fatou's lemma, lower semicontinuous (cf.\ \cite[Lemma 2.9]{r7}). Hence $J$ is a functional to which the critical point theory of \cite{r13} applies. 

\begin{definitions}[\cite{r12}, see also \cite{r13}] 
 \emph{ Let $X$ be a Banach space and $J=\Phi+\Psi$, where $\Phi \in C^1(X,\r)$ and $\Psi: X\to (-\infty,\infty]$ is lower semicontinuous and convex, $\Psi\not\equiv +\infty$. \\
(i)  The set $D(J):=\left\{u\in X: J(u)<+\infty\right\}$ is called the effective domain of $J$. \\
(ii) Let $u\in D(J)$. The set
 \begin{equation*}
\partial J(u):=\left\{w\in X^*: \langle \Phi'(u), v-u\rangle+\Psi(v)-\Psi(u)\geq \langle w, v-u\rangle\right\}  \quad\text{for all } v\in X \end{equation*}
 is called the subdifferential of $J$ at $u$. \\
(iii) $u\in X$ is a critical point of $J$ if $u\in D(J)$ and $0\in \partial J(u)$, i.e.
 \begin{equation*}
\langle\Phi'(u), v-u\rangle+\Psi(v)-\Psi(u)\geq 0 \quad\text{for all } v\in X.
 \end{equation*}
(iv) $(u_{n})$ is a Palais-Smale sequence for $J$ if $(J(u_{n}))$ is bounded and there exist $\varepsilon_{n}\rightarrow 0^{+}$ such that
 \begin{equation} \label{ps}
\langle\Phi'(u_{n}), v-u_{n}\rangle+\Psi(v)-\Psi(u_{n})\geq -\varepsilon_{n}\Vert v-u_{n}\Vert \quad\text{for all } v\in X.
\end{equation}
(v) $J$ satisfies the Palais-Smale condition if each Palais-Smale seqence has a convergent subsequence.
}
\end{definitions}

Below we summarize some properties of the functional $J$ given by \eqref{fcl}. $X$ will either denote the space \eqref{space} or $H^1(\rn)$, depending on the assumptions made on $V$. 

\begin{lemma}[{\cite[Lemma 2.2]{r12}}] \label{bdd}
If $\Omega\subset\rn$ is a bounded domain with regular boundary, then $J$ is of class $C^1$ in $H^1(\Omega)$.
\end{lemma}

The proof uses the fact that $|s\log s^2| \le C_p(1+|s|^{p-1})$, $p\in(2,2^*)$. 

\begin{proposition}[{\cite[Lemma 2.4]{r12}}] \label{subdiff}
If $u\in D(J)$, then there exists a unique $w\in X^*$ such that $\partial J(u)=\{w\}$, i.e.,
$$
\langle \Phi'(u),v-u\rangle +\Psi(v)-\Psi(u) \ge \langle w, v-u\rangle\quad\,\text{for this $w$ and all $v\in X$}. 
$$
Moreover, 
$$
\langle\Phi'(u),z\rangle+\irn F_1'(u)z\,dx = \langle w, z\rangle \quad \text{for all $z\in X$ such that $F_1'(u)z\in L^1(\rn)$}.
$$  
\end{proposition}

In \cite{r12} $V$ was periodic and $V+1$ positive but it is easy to see that the change of assumptions on $V$ does not affect the proof. 

The unique $w$ in Proposition \ref{subdiff} will be denoted by $J'(u)$.

\begin{lemma} \label{properties}
(i) If $u\in D(J)$, then $u$ is a solution of \eqref{1} if and only if $J'(u)=0$. \\
(ii) If $J(u_n)$ is bounded, then $J'(u_n)\to 0$ if and only if $(u_n)$ is a Palais-Smale sequence. \\
(iii) If $J(u_n)$ is bounded above, $J'(u_n)\to 0$ and $u_n\rh u$, then $u$ is a critical point of $J$.
\end{lemma}

\begin{proof}
(i) If $z\in C^\infty_0(\rn)$, then $F_1'(u)z\in L^1(\rn)$. So the conclusion follows from the second assertion of Proposition \ref{subdiff}. \\
(ii) Suppose $J(u_n)$ is bounded. According to \cite[Proposition 1.2]{r13}, the condition \eqref{ps} with $\eps_n\to 0^+$ is equivalent to the existence of a sequence $z_n\to 0$ such that
\[
\langle\Phi'(u_{n}), v-u_{n}\rangle+\Psi(v)-\Psi(u_{n})\geq \langle z_n,v-u_n\rangle \quad\text{for all } v\in X.
\]
(i.e., $z_n\in\partial J(u_n)$). Since here $z_n=J'(u_n)$, it follows that $J'(u_n)\to 0$ if and only if $(u_n)$ is a Palais-Smale sequence. \\
(iii) Since $\Psi$ is lower semicontinuous and convex, it is also weakly lower semicontinuous. So $\Psi(u)<\infty$ and $u\in D(J)$. Since $u_n\to u$ in $L^p_{loc}(\rn)$ for all $p\in[2,2^*)$, 
\begin{align*}
0 = \lim_{n\to\infty} \langle J'(u_n),v\rangle & = \lim_{n\to\infty} \left(\irn(\nabla u_n\cdot \nabla v + V(x)u_nv)\,dx - \irn u_nv\log u_n^2\,dx\right) \\
& = \irn(\nabla u\cdot \nabla v + V(x)uv)\,dx - \irn uv\log u^2\,dx = \langle J'(u),v\rangle
\end{align*}
for all $v\in C_0^\infty(\rn)$.
\end{proof}

Recall that $K$ denotes the set of critical points of $J$, i.e., $K = \{u\in D(J): J'(u)=0\}$. The following vector field of pseudo-gradient type will be important in what follows:

\begin{proposition}[{\cite[Lemma 2.7]{r12}}] \label{pseudogradient}
There exist a locally finite  countable covering $(W_j)$ of $D(J)\setminus K$, a set of points $(u_j)\subset D(J)\setminus K$ and a locally Lipschitz continuous vector field $H: D(J)\setminus K\to X$ with the following properties: \\
(i) The diameter of $W_j$ and the distance from $u_j$ to $W_j$ tend to 0 as $j\to\infty$. \\
(ii) $\|H(u)\|\le 1$ and $\langle J'(u),H(u)\rangle > z(u)$, where $z(u):= \min\frac12\|J'(u_j)\|$ for all $j$ such that $u\in W_j$. \\
(iii) $H$ has locally compact support, i.e.\ for each $u_0\in D(J)\setminus K$ there 
exist a neighbourhood $U_0$ of $u_0$ in $D(J)\setminus K$ and $R>0$ such that $\text{supp\,}H(u)\subset B_R(0)$ for all $u\in U_0$. \\
(iv) $J(u)>J(u_j)-\gamma_j$ for all $j$ such that $u\in W_j$, where $\gamma_j>0$ and  $\gamma_j\to 0$ as $j\to\infty$. \\
(v) $H$ is odd in $u$.
\end{proposition}

\begin{corollary} \label{pseudogradientcoroll}
For each $a\in\r$, $(W_j)$, $(u_j)$ and $H$ with the same properties as in Proposition \ref{pseudogradient} may be constructed on the set $\{u\in X: a < J(u) < \infty\}\setminus K$ (i.e., $D(J)\setminus K$ may be replaced by $\{u\in X: a < J(u) < \infty\}\setminus K$ throughout).
\end{corollary}

This follows by an easy inspection of the proof of \cite[Lemma 2.7]{r12}.

\begin{definition} \label{cpt-supp}
We shall say that a set $A\subset X$ has compact support if there exists $R>0$ such that $u(x)=0$ for all $|x|>R$ and $u\in A$.
\end{definition}

We shall also need a logarithmic Sobolev inequality \cite{r10} which holds for all $u\in H^{1}(\rn)$ and $a>0$:
\begin{equation} \label{logineq}
\irn u^{2}\log u^{2}\,dx \leq \frac{a^{2}}{\pi}\Vert \nabla u\Vert^{2}_{2}+\left(\log \Vert  u\Vert^{2}_{2}-N(1+\log a)\right)\Vert  u\Vert^{2}_{2}.
\end{equation}

\section{Proof of Theorem \ref{thm1}} \label{pf1}

\begin{proposition} \label{pscond}
If $(u_n)$ is a sequence such that $J(u_n)$ is bounded above and $J'(u_n)\to 0$, then $(u_n)$ has a convergent subsequence. In particular, $J$ satisfies the Palais-Smale condition.
\end{proposition}

\begin{proof}
First we show that $(u_n)$ is bounded. The proof is almost the same as that of \cite[Lemma 2.9]{r12} but for the reader's convenience we include it. Choose $d\in\r$ such that $J(u_n)\le d$ for all $n$. Then
\begin{equation} \label{eqa}
\Vert u_{n}\Vert_{2}^{2}=\irn u_{n}^{2}\,dx = 2J(u_{n})-\langle J'(u_{n}), u_{n}\rangle\leq 2d+o(1)\Vert u_{n}\Vert \quad\text{as } n\rightarrow \infty.
\end{equation}
Taking $a>0$ small enough in \eqref{logineq} gives
\begin{equation} \label{eqb}
\irn u^{2}\log u^{2}dx \leq \frac{1}{2}\Vert\nabla u \Vert_{2}^{2}+C_{1}\left(\log \Vert u \Vert_{2}^{2}+1\right)\Vert u \Vert_{2}^{2}.
\end{equation}
So using \eqref{eqa} and \eqref{eqb}  we obtain
\begin{align*}
2d & \geq 2J(u_{n})=\Vert u_{n}\Vert^{2}-\irn(V(x)+1)^{-}u_{n}^{2}\,dx-\irn u_{n}^{2}\log u_{n}^{2}\,dx \\
&\geq \frac{1}{2}\Vert u_{n}\Vert^{2}-C_{2}\left(\log \Vert u_{n} \Vert_{2}^{2}+1\right)\Vert u_{n} \Vert_{2}^{2} 
\geq \frac{1}{2}\Vert u_{n}\Vert^{2}-C_{3}\left(1+\Vert u_{n}\Vert^{r}\right),
\end{align*}
where we take $r\in(1,2)$. Hence the sequence $(u_{n})$ is bounded. 

Passing to a subsequence, $u_n\rh u$ in $X$ for some $u$ and since the embedding $X\hookrightarrow L^p(\rn)$ is compact for $p\in[2,2^*)$ as we have mentioned in the introduction, $u_n\to u$ in $L^p(\rn)$ for all such $p$. Taking $v=u$ in \eqref{ps} gives
\begin{align}
\langle u_n,u-u_n\rangle - \irn V^-(x)u_n(u-u_n)\,dx & - \irn F_2'(u_n)(u-u_n)\,dx \label{a1} \\
& + \Psi(u)-\Psi(u_n) \ge -\eps_n\|u-u_n\|, \nonumber
\end{align}
hence
\begin{equation} \label{a2}
\|u\|^2-\|u_n\|^2 + \Psi(u)-\Psi(u_n) + o(1) \ge o(1). 
\end{equation}
Since $\liminf_{n\to\infty}\Psi(u_n)\ge \Psi(u)$ and  $\liminf_{n\to\infty}\|u_n\|^2\ge \|u\|^2$, the inequality above implies $\|u_n\|\to\|u\|$ and hence $u_n\to u$ in $X$. 
\end{proof}

\begin{remark} \label{rem}
\emph{
In the next section $V$ will satisfy the assumptions of Theorem \ref{thm2}. The first part of the argument above shows, after a slight modification, that $(u_n)$ is bounded. Indeed, since the norm $\|\,\cdot\,\|$ is equivalent to the standard norm in $H^1(\rn)$, we obtain 
\begin{align*}
2d & \geq 2J(u_{n})=\Vert u_{n}\Vert^{2}-\irn u_{n}^{2}\log u_{n}^{2}\,dx 
\geq \frac{1}{2}\Vert u_{n}\Vert^{2}-C_{2}\left(\log \Vert u_{n} \Vert_{2}^{2}+1\right)\Vert u_{n} \Vert_{2}^{2}  \\
& \geq \frac{1}{2}\Vert u_{n}\Vert^{2}-C_{3}\left(1+\Vert u_{n}\Vert^{r}\right),
\end{align*}
provided $a$ in \eqref{logineq} is taken so small that $\frac12$ in \eqref{eqb} is replaced by a constant $b$ such that $b\|\nabla u\|_2^2\le \frac12\|u\|^2$. 
}
\end{remark}

We shall prove Theorem \ref{thm1} by adapting the arguments of Bartsch's Fountain Theorem \cite[Theorem 2.5]{ba}, \cite[Theorem 3.6]{r16}  to our situation. First we prove a suitable deformation result.

\begin{lemma} \label{noncrit}
If $K_{d}=\emptyset$, then there exists $\eps_{0}>0$ such that there are no Palais-Smale sequences in $J^{d+2\eps_{0}}_{d-2\eps_{0}}$.
\end{lemma}

\begin{proof}
Arguing by contradiction we find a sequence $(u_{n})$ such that $J(u_{n})\rightarrow d$ and $J'(u_n)\to 0$. According to Proposition \ref{pscond} and Lemma \ref{properties}(iii), $u_{n}\rightarrow u$ and $u\in K$, possibly after taking a subsequence. We shall show that $J(u)=d$. 
 Since $(u_n)$ is a Palais-Smale sequence,
\begin{equation*}
\langle \Phi'(u_{n}), u-u_{n} \rangle+\Psi(u)-\Psi(u_{n})\geq -\eps_n\|u-u_n\|, \quad \eps_n\to 0^+,
\end{equation*}
and thus $\limsup_{n\to\infty}\Psi(u_n)\le \Psi(u)$. So it follows from the lower semicontinuity of $\Psi$ that $\Psi(u_n)\to \Psi(u)$. Hence $J(u_n)\to J(u)$ and $J(u)=d$.
As $u\in K$, this contradicts the hypothesis $K_{d}=\emptyset$.
\end{proof}

Suppose $K_d=\emptyset$ and let $\eps_{0}$ be as in Lemma \ref{noncrit}. Let  $\chi: X\rightarrow [0, 1]$ be an even locally Lipschitz continuous function such that $\chi=0$ on $J^{d-\eps_0}$ and $\chi>0$ otherwise. Consider the flow $\eta$ given by
\begin{equation} \label{flow}
\left\{
\begin{array}{l}
\frac{d}{dt}\eta(t, u)=-\chi(\eta(t, u))H(\eta(t, u)), \\ [2pt]
 \eta(0, u)=u, \quad u\in J^{d+\eps_0},
\end{array}
\right.
\end{equation}
where $H: \{u\in X: d-2\eps_0<J(u)<\infty\}\setminus K\to X$ is the vector field as in Corollary \ref{pseudogradientcoroll} and $\chi(u)H(u)$ should be understood as 0 if $J(u)\le d-\eps_0$. It has been shown in \cite{r12} (see the beginning of the proof of Lemma 2.13 there) that $t\mapsto J(\eta(t,u))$ is differentiable and $\frac{d}{dt}J(\eta(t, u)) = \left\langle J'(\eta(t, u)), \frac{d}{dt}\eta(t,u) \right\rangle$. So
\begin{equation*}
\frac{d}{dt}J(\eta(t, u))=-\langle J'(\eta(t, u)), \chi(\eta(t, u))H(\eta(t, u))\rangle\leq -\chi(\eta(t, u))z(\eta(t, u))\leq 0,
\end{equation*}
thus $t\mapsto J(\eta (t, u)$ is nonincreasing and, since $\|H(u)\|\le 1$ and $K\cap J^{d+\eps_{0}}_{d-\eps_{0}}=\emptyset$, $\eta(t,u)$ exists for all $t\ge 0$. 

\begin{proposition}[Deformation] \label{deformation} 
Suppose $K_{d}=\emptyset$ and $\eps_{0}$ is as in Lemma \ref{noncrit}. If $\eps\in(0,\eps_0)$, then for each compact set $A\subset J^{d+\eps}\cap C_0^\infty(\rn)$  there exists $T>0$ such that $J(\eta(T,A))\subset J^{d-\eps}$. 
\end{proposition}

\begin{proof}
First we show that $\tau := \inf\{z(u): u\in J_{d-\eps_0}^{d+\eps_0}\} > 0$. Suppose $\tau=0$. Then we can find a sequence $w_{n}\in J^{d+\eps_0}_{d-\eps_0}$ such that $z(w_{n})\rightarrow 0$. By Corollary \ref{pseudogradientcoroll} there exists a sequence $(u_{n})$ such that $J'(u_{n})\rightarrow 0$ and $d-2\eps_{0}< J(u_{n})\leq J(w_{n})+\gamma_{n}\leq d+2\eps_{0}$ for $n$ sufficiently large. So $(u_n)$ is a Palais-Smale sequence, a contradiction to Lemma \ref{noncrit}.

Consider the auxiliary flow
\[
\left\{
\begin{array}{l}
\frac{d}{dt}\sigma(t, u)=-H(\sigma(t, u)), \\ [2pt]
 \sigma(0, u)=u, \quad u\in J^{d+\eps_0}_{d-\eps_0}.
\end{array}
\right.
\]
For each $u$ as above this flow exists as long as $J(\sigma(t,u)) > d-2\eps_0$. So
\begin{align*}
J(\sigma(t,u))-J(u) = -\int_0^t\langle J'(\sigma(s,u),H(\sigma(s,u))\rangle\,ds \le -\int_0^t z(\sigma(s,u))\,ds 
\le -\tau t
\end{align*}
and therefore $J(\sigma(t,u)) \le J(u)-\tau t \le d+\eps_0 -\tau t\le d-\eps_0$ if $t\ge 2\eps_0/\tau$. In particular, $\sigma(\cdot,u)$ must enter the set $J^{d-\eps_0}$. 

Let now $A\subset J^{d+\eps}\cap C_0^\infty(\rn)$. If $u\in J^{d-\eps_0}$, then $J(\eta(t,u))=u$ for all $t\ge 0$. On the set $d-\eps_0<J(u)\le d+\eps_0$, $\eta$ and $\sigma$ have the same flow lines. So in particular, $J(\eta(t,u)) < d-\eps$ for some $t>0$. Since 
\[
\eta(t,u) = u - \int_0^t\chi(\eta(s, u))H(\eta(s, u))\,ds
\]
and $H$ has locally compact support, it is easy to see that $\eta([0,t]\times A)$ has compact support (in the sense of Definition \ref{cpt-supp}) for each $t$. Given $u_0\in A$, we can find $T_0>0$ such that $J(\eta(T_0,u_0)) < d-\eps$. According to Lemma \ref{bdd}, the restriction of $J$ to $\eta([0,T_0]\times A)$ is continuous, hence $J(\eta(T_0,u))<d-\eps$ for all $u$ in a neighbourhood $A_0$ of $u_0$ in $A$. Using compactness of $A$ we find a finite covering $(A_j)$ of $A$ and $T_j>0$ such that $J(\eta(T_j,u)) < d-\eps$ for all $u\in A_j$. Taking $T := \max_j T_j$, the conclusion follows.
\end{proof}

\begin{remark} \label{remdef}
\emph{
For the purpose of the next section let us note that in the proof above we have not used the Palais-Smale condition but only the fact that there are no Palais-Smale sequences in $J^{d+2\eps_0}_{d-2\eps_0}$. 
}
\end{remark}

Since $X$ is separable and $C_{0}^{\infty}(\rn)$ is dense in $X$, there exists a sequence $(X_k)\subset C_0^\infty(\rn)$ of subspaces such that $\dim X_k=k$ and $X=\ol{\cup_{k=1}^\infty X_k}$. Let $Z_k := X_{k}^\bot$ and
\begin{equation*}
B_{k}:=\{u\in X_{k}: \Vert u\Vert\leq \rho_{k}\},\quad N_{k}:=\{u\in Z_{k-1}: \Vert u\Vert= r_{k}\}, \quad \text{where } \rho_{k}>r_{k}>0.
\end{equation*}

\begin{lemma} [{\cite[Lemma 3.4]{r16}}] \label{intersection}
If $\gamma\in C(B_{k}, X)$ is odd and $\gamma|_{\partial B_{k}}=id$, then $\gamma(B_{k})\cap N_{k}\neq \emptyset$.
\end{lemma}

\begin{proposition} \label{fountain}
There exist $\rho_{k}>r_{k}>0$ such that
\[
a_{k}:= \max_{\substack{u\in X_k \\ \|u\|=\rho_k}}J(u) \le 0 \text{ for all } k \quad \text{and} \quad b_k := \inf_{u\in N_k} J(u) \to\infty \text{ as } k\to\infty.
\]
\end{proposition}

\begin{proof}
Let $u=sw$, where $u\in X_k$ and $\|w\|_2=1$. Then 
\[
J(sw) = \frac{s^2}2\left(\irn\left(|\nabla w|^2+(V(x)+1)w^2\right)dx -\log s^2 - \irn w^2\log w^2\,dx\right).
\]
Since all norms in $X_k$ are equivalent and $X_k\subset C_0^\infty(\rn)$, both integrals above are uniformly bounded. Hence $J(sw)\to-\infty$ uniformly in $w$ as $s\to\infty$, so there exists $\rho_k$ such that $a_k\le 0$. Moreover, $\rho_k$ may be chosen as large as we need.

Let 
\[
\beta_{k} := \sup_{\substack{u\in Z_{k-1}\\ \|u\|=1}}\|u\|_2.
\]
Then $\beta_k\to 0$. The proof is the same as in \cite[Lemma 3.8]{r16} but we include it for the reader's convenience. The sequence $(\beta_k)$ is positive and decreasing, hence $\beta_k\to\beta\ge 0$ and $\|u_k\|_2\ge \beta_k/2$ for some $u_k\in Z_{k-1}$, $\|u_k\|=1$. Since $u_k\rh 0$ in $X$, $u_k\to 0$ in $L^2(\rn)$. This implies that $\beta=0$. 

Using  \eqref{eqb} as in the proof of Proposition \ref{pscond} we obtain
\begin{align*}
J(u) & =\frac12\Vert u\Vert^{2}-\frac12\irn(V(x)+1)^{-}u^{2}\,dx-\frac12\irn u^{2}\log u^{2}\,dx \\
&\geq \frac{1}{4}\Vert u\Vert^{2}-C_1\left(\log \Vert u \Vert_{2}^{2}+1\right)\Vert u \Vert_{2}^{2} 
\geq \frac{1}{4}\Vert u\Vert^{2}-C_{2}\|u\|_2^{p}-C_3 \\
& \ge \frac14\|u\|^2-C_2\beta_k^p\|u\|^p-C_3, 
\end{align*}
where $p\in(2,2^*)$. Let $r_k=1/\beta_k$ and $\|u\|=r_k$. Then
\[
J(u) \ge \frac1{4\beta_k^2} -C_2-C_3 \to \infty \quad \text{as } k\to\infty
\]
and hence $b_k\to\infty$. Since we may choose $\rho_k>r_k$, the proof is complete.
\end{proof}

\begin{proof}[Proof of Theorem \ref{thm1}]
Let 
\[
\Gamma_{k} := \left\{\gamma\in C(B_{k}, X): \gamma\text{ is odd, } \gamma|_{\partial B_{k}}=id \text{ and } \gamma(B_{k})\text { has compact support}\right\}
\]
and
\[
d_{k}:=\inf_{\gamma\in \Gamma_{k}}\max_{u\in B_{k}} J(\gamma (u)).
\]
Since $\gamma(B_k)\cap N_k\ne\emptyset$ according to Lemma \ref{intersection}, $d_k\ge b_k \to\infty$ and it remains to show that $K_{d_k}\ne\emptyset$ if $k$ is large. Assuming the contrary, choose $\eps_0, \eps$ and $T$ as in Proposition \ref{deformation} and let $\gamma\in\Gamma_k$ be such that $\gamma (B_k) \subset J^{d_k+\eps}$. Let $\beta(u) := \eta(T,\gamma(u))$, where $\eta $ is the flow \eqref{flow}. Since $\eta(T,u) = u$ for all $u\in J^{d_k-\eps_0}$, $\beta\in\Gamma_k$. By Proposition \ref{deformation}, $\beta(B_k)\subset J^{d_k-\eps}$, a contradiction to the definition of $d_k$. 
\end{proof}

\begin{remark} \label{cptsupp}
\emph{
It was important for the argument of Proposition \ref{deformation} that the set $A$ has compact support and the flow $\eta$ has the property that $\eta([0,t],A)$ has compact support for each $t\ge 0$. Without this it is not clear whether $T$ as required can be found ($T_0$ in the proof of this proposition exists for each $u_0$ but a neighbourhood $A_0$ on which $J(\eta(T_0),u) < d-\eps$ may not exist because of the lack of continuity of $J$). 
}

\emph{
We would like to point out here that the argument of Theorem 1.1 in \cite{r12} contains a minor gap because the property of compact support is not assumed there. This gap can be removed by requiring that the class $\mathcal{H}$ of mappings introduced there takes sets having compact support into sets with the same property and then minimaxing over compact sets with compact support. Also the proof of Theorem 1.2 in \cite{r12} requires a small modification: the paths in the Mountain Pass argument need to be approximated by paths having compact support. This can be done in the same way as in the proof of Theorem \ref{thm2} below. 
}
\end{remark}

\section{Proof of Theorem \ref{thm2}} \label{pf2}

Here we work in the space $X = H^1(\rn)$ and the functional \eqref{fcl} can be written in the form
\[
J(u) = \frac12\|u\|^2-\frac12\irn u^2\log u^2\,dx.
\] 

\begin{proposition} \label{boundedness}
If $(u_n)$ is a sequence such that $J(u_n)$ is bounded above and $J'(u_n)\to 0$, then $(u_n)$ is bounded. In particular, $(u_n)$ is a Palais-Smale sequence.
\end{proposition}

\begin{proof}
This follows from Proposition \ref{pscond} in view of Remark \ref{rem}. 
\end{proof}

We shall need a limiting problem
 \begin{equation}  \label{lim}
-\Delta u+V_{\infty}u =u \log u^{2}, \quad x\in \rn.
\end{equation}
The energy functional corresponding to it is  
\[
J_{\infty}(u) := \frac12\irn\left(|\nabla u|^2+(V_{\infty}+1)\right)dx - \frac12\irn u^2\log u^2\ dx. 
\]   
Let 
\[
\cn := \{ u\in D(J)\setminus\{0\}: \langle J'(u),u\rangle = 0\}
\]
be the Nehari manifold for $J$ and define the Nehari manifold $\cn_\infty$ for $J_\infty$ in the same way. By \cite[Theorem 1.2]{r12} there exists a solution $u_\infty>0$ for \eqref{lim} which minimizes $J_\infty$ on $\cn_\infty$ (a ground state solution). It is easy to see (cf.\ \cite{r12}) that if $u\in D(J)\setminus \{0\}$ and $\vp_u(s):=J(su)$, $s>0$, then 
\[
\vp_u'(s) = s\left(\|u\|^2 - \irn u^2(\log s^2+\log u^2+1)\,dx\right) = 0
\]
for a unique $s$, and this is the unique intersection point of the ray $\{su: s>0\}$ with $\cn$. Moreover, $\vp_u(s)\to-\infty$ as $s\to\infty$ and if $\|u\|=1$, then $s\mapsto \Phi(su)$ increases for $0<s<s_0$ ($s_0$ independent of $u$) and $s\mapsto \Psi(su)$ increases for all $s>0$ (by convexity). Hence $\cn$ is bounded away from the origin.

Let 
\[
\Gamma:=\{\alpha\in C([0, 1], X): \alpha(0)=0,\ J(\alpha(1))<0\}
\]
and 
\[
c := \inf_{\alpha\in\Gamma}\sup_{s\in[0,1]} J(\alpha(s)), \quad c_\cn := \inf_{u\in \cn} J(u).
\]
Clearly, $c\le c_\cn$, and since $F_2'(s)\le C|s|^{p-1}$ and $\Psi\ge 0$, it is easy to see that $J(u)$ is bounded away from 0 on a (small) sphere around the origin (cf.\ \cite[Lemma 2.15]{r12}). In particular, $c>0$ (and $c$ is the mountain pass level). 

\begin{lemma} \label{compar}
(i) If $V\not\equiv V_\infty$, then $c_\cn < c_\infty$, where $c_\infty := \inf_{u\in\cn_\infty}J_\infty(u)$. \\
(ii) If $J(u_n)\to d\in (0, c_\infty)$ and $J'(u_n)\to 0$, then $u_n\rh u\ne 0$ after passing to a sub\-sequence, $u$ is a critical point of $J$ and $J(u)\le d$.
\end{lemma}

\begin{proof}
(i) Let $s_0>0$ be such that $s_0u_\infty\in \cn$ where $u_\infty>0$ is a ground state for \eqref{lim}. Since $V(x) < V_\infty$ in an open set, $u_\infty>0$ and $s\mapsto J_\infty(su_\infty)$, $s>0$, has a unique maximum at $s=1$,
\[
c_\cn \le J(s_0u_\infty) < J_\infty(s_0u_\infty) \le J_\infty(u_\infty) = c_\infty.
\]
(ii) By Proposition \ref{boundedness}, $u_n\rh u$ in $X$, $u_n(x)\to u(x)$ a.e.\ after passing to a subsequence and, according to (iii) of Lemma \ref{properties}, $u$ is a critical point of $J$. By Fatou's lemma,
\begin{align*}
d & = J(u_n) -\frac12\langle J'(u_n),u_n\rangle +o(1) = \frac12\irn u_n^2\,dx + o(1) \\
& \ge \frac12\irn u^2\,dx + o(1) = J(u) - \frac12\langle J'(u),u\rangle +o(1) = J(u)+o(1).
\end{align*}
So $J(u)\le d$ and it remains to show that $u\ne 0$. Arguing indirectly, suppose $u=0$. Since $u_n\to 0$ in $L^2_{loc}(\rn)$ and $V(x)\to V_\infty$ as $|x|\to\infty$, 
\[
J(u_n)-J_\infty(u_n) = \frac12\irn(V(x)-V_\infty)u_n^2\,dx \to 0
\]
and therefore $J_\infty(u_n)\to d$. 
Using the H\"older and the Sobolev inequalities and taking $v$ with $\|v\|=1$, we obtain 
\begin{align*}
\left|\langle J'(u_n)-J'_\infty(u_n), v\rangle \right| & \le \irn (V_\infty-V(x))|u_n|\,|v|\,dx \\
& \le C\left(\irn(V_\infty-V(x))u_n^2\,dx\right)^{1/2}.
\end{align*}
As the right-hand side tends to 0 uniformly in $\|v\|=1$, $J'(u_n)-J'_\infty(u_n)\to 0$ and hence $J'_\infty(u_n)\to 0$. Taking $p\in(2,2^*)$, we have
\begin{equation*}
 o(1) = \langle J'(u_n),u_n\rangle \geq \|u_n\|^2-C_1\int_{\{u_n^2\geq 1/e\}}  |u_n|^{p}\,dx,
 \end{equation*} 
and if $\|u_n\|_p\to 0$, then $u_n\to 0$ in $X$. By \eqref{ps} with $v=0$, 
\[
\langle \Phi'(u_n),-u_n\rangle -\Psi(u_n) \ge -\eps_n\|u_n\|
\]
which implies $\Psi(u_n)\to 0$. So $J(u_n)\to 0$ contrary to the assumption that $J(u_n)\to d>0$. It follows that $\|u_n\|_p\not\to 0$ and  
hence by Lions' lemma \cite[Lemma I.1]{li}, \cite[Lemma 1.21]{r16} there are $(y_n)\subset\rn$ and $\delta>0$ such that for large $n$,
\[
\int_{B_1(y_n)}u_n^2\,dx \ge \delta. 
\]
Let $v_n(x) := u_n(x+y_n)$. Since $J_\infty$ is invariant with respect to translations by elements of $\rn$, $J_\infty(v_n)\to d$ and $J'_\infty(v_n) \to 0$. Moreover,
\[
\int_{B_1(0)}v_n^2\,dx = \int_{B_1(y_n)}u_n^2\,dx \ge \delta 
\]
and therefore $v_n\rh v\ne 0$ after passing to a subsequence. So $v$ is a nontrivial critical point of $J_\infty$ and $J_\infty(v)\le d < J_\infty(u_\infty)$ which is the desired contradiction.
\end{proof}

\begin{proof}[Proof of Theorem \ref{thm2}]
If $V\equiv V_\infty$, then $u_\infty$ is a solution we are looking for. So assume $V(x)<V_\infty$ for some $x$.

Suppose there exists $\eps_0\in (0,c/2)$ such that  there are no Palais-Smale sequences in $J_{c-2\eps_0}^{c+2\eps_0}$ and let $\eps\in(0,\eps_0)$. Choose $\alpha\in\Gamma$ so that $\alpha([0,1])\subset J^{c+\eps/2}$. We may assume $J(\alpha(1))<-\eps/2$. Let $\chi_{R}\in C^1(\rn, [0, 1])$ be such that $\chi_{R}(x)=1$ if $\vert x\vert\leq R$, $\chi_{R}(x)=0$ if $\vert x\vert\geq 2R$ and $\vert \nabla \chi_{R}\vert\leq 1$.
Let $u_{R}(x):=\chi_{R}(x)u(x)$. It is easy to see that $\Vert u_{R}-u\Vert\rightarrow 0$ uniformly in $u\in \alpha([0,1])$ as $R\rightarrow +\infty$. Since $\Phi\in C^{1}(X, \r)$, there exists $R>0$ such that $\Phi(u_R)\le \Phi(u)+\eps/2$ for all $u\in \alpha([0,1])$. Moreover, as $F_1$ is convex and $|u_R|\le |u|$, $\Psi(u_R)\le \Psi(u)$. So it follows that $\alpha_R([0,1])\subset J^{c+\eps}$ and $J(\alpha_R(1))<0$, where $\alpha_R(s) := \chi_R\alpha(s)$. Clearly, $\alpha_R$ has compact support and $\alpha_R(0)=0$. Hence $\alpha_R\in\Gamma$.
Using Proposition \ref{deformation} and Remark \ref{remdef} we set $\beta_R(s) := \eta(T,\alpha_R(s))$ and obtain $\beta_R\in \Gamma$, $\beta_R([0,1])\subset J^{c-\eps}$, a contradiction to the definition of $c$. Since $\eps_0$ may be chosen arbitrarily small, there exists a sequence $(u_n)$ such that $J'(u_n)\to 0$ and $J(u_n)\to c$. Using Lemma \ref{compar} we obtain a critical point $u\ne0$ of $J$ such that $J(u)\le c$. So $u\in \cn$. Hence $c=c_\cn$ and $u$ is a ground state solution (and so is $-u$). As in \cite{r12}, we first see that $\pm u$ cannot change sign and then that either $u$ or $-u$ is strictly positive. This completes the proof. 
\end{proof}

\section{Extension to the $p$-Laplacian} \label{plapl}

For the equation
 \begin{equation}  \label{2}
-\text{div}(\vert \nabla u\vert^{p-2}\nabla u)+V(x)\vert u\vert^{p-2}u =\vert u\vert^{p-2} u \log \vert u\vert^{p}, \quad x\in \rn,
\end{equation}
where $1<p<N$, results similar to Theorems \ref{thm1} and \ref{thm2} hold. The functional corresponding to \eqref{2} is
\begin{equation} \label{p-fcl} 
J(u):= \frac{1}{p}\irn\left(\vert \nabla u\vert^{p}+(V(x)+1)|u|^{p}\right)dx-\frac{1}{p}\irn |u|^{p}\log |u|^{p}\,dx.
\end{equation}

\begin{theorem} \label{thm3}
If $V\in C(\rn,\r)$ and $\lim_{|x|\to\infty}V(x)=\infty$, then equation \eqref{2} has infi\-nitely many solutions $\pm u_n$ such that $J(\pm u_n)\to\infty$.
\end{theorem}

\begin{theorem} \label{thm4}
If $V\in C(\rn,\r)$, $\lim_{|x|\to\infty}V(x) = \sup_{x\in\rn}V(x) := V_\infty \in (-1,\infty)$ and 
\begin{equation} \label{assu}
\inf\left\{\irn\left(\vert \nabla u\vert^{p}+(V(x)+1)|u|^{p}\right)dx: \|u\|_p=1\right\} > 0,
\end{equation}
then equation \eqref{1} has a ground state solution $u>0$. 
\end{theorem}

Since the proofs are similar to those of Theorems \ref{thm1} and \ref{thm2}, we only point out the main differences. 

In Theorem \ref{thm3} we shall work in the space 
\begin{equation} \label{space2}
X := \{u\in W^{1,p}(\rn): \|u\|^p := \irn(|\nabla u|^p+(V(x)+1)^+|u|^p)\,dx < \infty\}.
\end{equation}
It is well known that the embedding $X\hookrightarrow L^p(\rn)$ (and hence $X\hookrightarrow L^q(\rn)$ for all $p\le q< p^* := Np/(N-p)$) is compact. Since we could not find any convenient reference to this result, we include a brief argument. Let $(u_n)$ be bounded and let $\eps>0$ be given. Passing to a subsequence, $u_n\rh u$ and $\|u_n-u\|^p \le C$ for some $C$. Choose $R>0$ so that $V(x)+1 \ge C/\eps$ if $|x|>R$. Then
\[
\int_{|x|>R}|u_n-u|^p\,dx \le \frac{\eps}C\irn(V(x)+1)^+|u_n-u|^p\,dx \le \frac{\eps}C\|u_n-u\|^p \le \eps
\]
and since $u_n\to u$ in $L^p_{loc}(\rn)$, the conclusion follows. 

The results of Section \ref{prel}, including the definition of the functionals $\Phi$ and $\Psi$, remain valid after making obvious changes (in particular,  in the definition of $F_1$, $-\frac12s^2\log s^2$ should be replaced by $-\frac1p|s|^p\log|s|^p$ for $|s|<\delta$).  In the proof of the boundedness part of Proposition \ref{pscond}  inequality \eqref{logineq} must be replaced by a $p$-logarithmic Sobolev inequality. A suitable version for our purposes is
\[
\irn|u|^p\log|u|^p\,dx \le \frac Np\log(C_p\|\nabla u\|_p^p),
\]
see the formula in the middle of p.\ 153 in \cite{r8}. In the convergence part of the proof of Proposition \ref{pscond}, a formula corresponding to \eqref{a1} is
\begin{align*}
& \irn \left(|\nabla u_n|^{p-2}\nabla u_n\cdot \nabla(u-u_n)+(V(x)+1)^+|u_n|^{p-2}u_n(u-u_n)\right)dx  \\
& \quad - \irn(V(x)+1)^-|u_n|^{p-2}u_n(u-u_n)\,dx - \irn F_2'(u_n)(u-u_n)\,dx \\
& \quad + \Psi(u)-\Psi(u_n)
\ge -\eps_n\|u-u_n\|, 
\end{align*} 
and this gives \eqref{a2} with exponent 2 replaced by $p$.
In Lemma \ref{noncrit} and Proposition \ref{deformation} only minor changes are needed. On the other hand, the construction of  $(X_k)$ and $(Z_k)$ requires some care. It is known that the space $W^{1,p}(\rn)$ has a Schauder basis $(e_k)_{k=1}^\infty$, see e.g.\ \cite[Section 2.5.5, Remark 2]{tr}. By \cite[Proposition 1.a.9]{lt}, this basis may be chosen so that $e_k\in C_0^\infty(\rn)$ for all $k$. There exists a set of biorthogonal functionals $(e_k^*)_{k=1}^\infty\subset X^*$, i.e. functionals such that $\langle e_m^*,e_k\rangle = \delta_{km}$ \cite[Section 1.b]{lt}. It is then easy to see that $e_k^*$ are total, i.e.\ $\langle e_k^*,u\rangle = 0$ for all $k$ implies $u=0$ (cf.\ \cite[Section 1.f]{lt}). Now we take $X_k := \text{span}\{e_1,\ldots,e_k\}$ and $Z_k := \text{cl\,span}\{e_{k+1},e_{k+2},\ldots\}$, where cl denotes the closure, and we define $a_k, b_k, B_k, N_k$ as previously. Lemma \ref{intersection} still holds, essentially with the same proof as in \cite{r16}, see \cite[Proposition 4.6]{sw}. In the proof of Proposition \ref{fountain} the exponent 2 should be replaced by $p$ and $\beta_k^p\|u\|^p$ by $\beta_k^q\|u\|^q$, where $q\in(p,p^*)$. That $\beta_k\to 0$ follows from the fact that if $u_k\in Z_{k-1}$ and $u_k\rh u$, then $u$ must be 0 because $e_k^*$ are total. The remaining part of the proof is unchanged. 

In Theorem \ref{thm4} we work in the space $X = W^{1,p}(\rn)$ with the same norm as in \eqref{space2}.
Since $V_\infty+1>0$, it is equivalent to the usual norm. Note however that if $V(x)+1<0$ for some $x$, then the integral in \eqref{space2} with $(V(x)+1)^+$ replaced by $V(x)+1$ does not define a norm. The rest of the argument is essentially the same. Assumption \eqref{assu} implies that each ray intersects the Nehari manifold at a unique point and since $ u\mapsto\irn(V(x)+1)^-|u|^p\,dx$ is weakly continuous, one easily verifies that $J$ has the mountain pass geometry. Finally, it follows from \cite[Theorem 5.3.1]{ps} that if $u\ge 0$ is a nontrivial solution, then $u>0$. More precisely, if we take $f(u)=u^{p-1}(C-\log u^p)$ with $C>0$ large enough, then $f$ is increasing on $(0,\eps)$ for some $\eps>0$ and $f(u) \ge u^{p-1}(V(x)-\log u^p)$ for all $u > 0$. Let $F(u) := \int_0^u f(s)\,ds$. Since $F(u)^{-1/p}$ is not integrable at $u=0$, all assumptions of (the first part of) the above-mentioned theorem are satisfied and hence $u>0$.

\end{document}